\documentclass{amsart}
\usepackage{amsfonts,amssymb,latexsym,amsmath, amsxtra}
\usepackage[all]{xy}
\usepackage[dvips]{graphics}
\usepackage{hyperref}

\usepackage[OT2,T1]{fontenc}
\DeclareSymbolFont{cyrletters}{OT2}{wncyr}{m}{n}
\DeclareMathSymbol{\Sha}{\mathalpha}{cyrletters}{"58}

\theoremstyle{plain}
\newtheorem{theorem}{Theorem}[section]
\newtheorem{corollary}[theorem]{Corollary}

\newtheorem{lemma}[theorem]{Lemma}
\newtheorem{proposition}[theorem]{Proposition}
\newtheorem*{conjecture*}{Conjecture}

\theoremstyle{definition}

\theoremstyle{remark}
\newtheorem*{remark}{Remark}
\newtheorem*{remarks}{Remarks}

\numberwithin{equation}{section}

\newcommand{\R}{\mathbb R}
\newcommand{\N}{\mathbb N}
\newcommand{\Z}{\mathbb Z}
\newcommand{\C}{\mathbb C}

\def\PP{\mathbb{P}}

\def\H{\mathbb H}

\def\SL{\rm SL}

\def\sgn{\rm sgn}
\def\({\left(}
\def\){\right)}

\newcommand{\dx}{\mathrm{d}x}

\newcommand{\erf}{\mathrm{erf}}

\newcommand{\abs}[1]{\left|#1\right|}

\newcommand{\II}{\mathcal{I}}

\def\k2{\frac{k}{2}}

\begin{document}

\title{Faber Polynomials and Poincar\'e series}

\author{Ben Kane}
\address{Mathematical Institute\\University of
Cologne\\ Weyertal 86-90 \\ 50931 Cologne \\Germany}
\email{bkane@math.uni-koeln.de}

\subjclass[2000] {11N37, 11F37, 11F11, 26A33, 11J91}
\keywords{Faber polynomials, Poincar\'e series, harmonic weak Maass forms, special functions, Gauss error function}

\date{\today}
\thispagestyle{empty} \vspace{.5cm}

\begin{abstract}
In this paper we consider weakly holomorphic modular forms (i.e. those meromorphic modular forms for which poles only possibly occur at the cusps) of weight $2-k\in 2\Z$ for the full modular group $\SL_2(\Z)$.  The space has a distinguished set of generators $f_{2-k,m}$.  Such weakly holomorphic modular forms have been classified in terms of finitely many Eisenstein series, the unique weight 12 newform $\Delta$, and certain Faber polynomials in the modular invariant $j(z)$, the Hauptmodul for $\SL_2(\Z)$.  We employ the theory of harmonic weak Maass forms and (non-holomorphic) Maass-Poincar\'e series in order to obtain the asymptotic growth of the coefficients of these Faber polynomials.  Along the way, we obtain an asymptotic formula for the partial derivatives of the Maass-Poincar\'e series with respect to $y$ as well as extending an asymptotic for the growth of the $\ell$-th repeated integral of the Gauss error function at $x$ to include $\ell\in \R$ and a wider range of $x$.
\end{abstract}
\maketitle

\section{Introduction}

Let $S_k$ be the space of weight $k\in 2\Z$ cusp forms for the full modular group $\SL_2(\Z)$.  The first case where $S_k$ is non-empty is $k=12$.  Let $\Delta(z)\in S_{12}$ be the unique normalized weight $12$ cusp form (newform) for the full modular group $\SL_2(\Z)$.  Following Ramanujan, we denote the Fourier coefficents of $\Delta$ by $\tau(n)$ and refer to $\tau:\N\to\Z$ as \begin{it}Ramanujan's tau function\end{it}, so that 
$$
\Delta(z)=\sum_{n\geq 1} \tau(n) q^n,
$$
where $q=e^{2\pi i z}$.  Since $\Delta(z)$ does not vanish on the upper half plane, inverting $\Delta$ leads naturally to the study of \begin{it}weakly holomorphic modular forms\end{it}, that is, those modular forms which are holomorphic on the upper half plane but which are only meromorphic at the (unique) cusp $\infty$.  We denote the space of weight $2-k$ weakly holomorphic modular forms on $\SL_2(\Z)$ by $M_{2-k}^!$.  For $k\geq 2$, let $d:=d_k$ denote one less than the dimension of the space of holomorphic modular forms, so that $d_k=\dim\left(S_{k}\right)$ when $k\neq 2$ and $d_2=-1$.  There is a distinguished set of generators $f_{2-k,m}\in M_{2-k}^!$ ($m\in \Z$) which satisfy
\begin{equation}\label{eqn:fmdef}
f_{2-k,m}=q^{-m}+O\left(q^{-d}\right),
\end{equation}
and moreover $f_{2-k,m}$ is unique among weakly holomorphic modular forms satisfying \eqref{eqn:fmdef}.  The $f_{2-k,m}$ are natural in a number of ways.  When $k=2$, this set plays a central role in the study of singular moduli \cite{Zagiersingular} and is closely knit to the Hecke operators \cite{AsaiKanekoNinomiya}.  By work of Duke and Jenkins \cite{DukeJenkins08}, there is also a duality which relates the $n$-th Fourier coefficient of $f_{2-k,m}$ to the $m$-th Fourier coefficient of $f_{k,n}$, paralleling the duality in the weight $\frac{3}{2}$ case famously obtained by Zagier \cite{Zagiersingular} while giving a new proof of Borcherds' identity.  In the case $k=2$, $f_{0,m}$ have also been shown to satisfy interesting congruences under the $U(p)$-operator.  For example, Lehner \cite{Lehner} proved that $f_{0,1}|U(p)$ is congruent to a constant $\pmod{p}$ whenever $p\leq 11$, while Serre \cite{Serre} has shown that
$$
f_{0,1}|U(13) \equiv -\Delta(z)\pmod{13}
$$
and that $f_{0,m}|U(p)$ is never congruent to a constant $\pmod{p}$ whenever $p\geq 13$.  Elkies, Ono, and Yang \cite{EOY} have recently considered the more general question of whether linear combinations of $f_{0,m}|U(p)$ can be congruent $\pmod{p}$ to a linear combination of other $f_{0,m'}$ by relating this question to the study of supersingular $j$-invariants.

An inspection of the set $\left\{ f_{2-k,m}\big|m> d\right\}$ leads one naturally to a study of \begin{it}generalized Faber polynomials\end{it}, first defined when $k=2$ by Faber in \cite{Faber} and generalized in \cite{Faber2}, which in the case $k=2$ are related to the denominator formula for the Monster Lie algebra.  Indeed, these weakly holomorphic modular forms are explicitly constructed by Duke and Jenkins \cite{DukeJenkins08} as
\begin{equation}\label{eqn:fmexplicit}
f_{2-k,m}(z) := \left \{ \begin{array}{cl} 
E_{k'}(z) \Delta(z)^{-d-1} F_m( j(z)) & \text{ if }m > d,\\
0 & \text{ if } m\leq d,
\end{array} \right.
\end{equation}
where $k'\in \{ 0,4,6,8,10,14\}$ with $k'\equiv 2-k\pmod{12}$, $E_{k'}$ is the Eisenstein series of weight $k'$, and $F_m$ is a generalized Faber polynomial of degree $m-d-1$ chosen recursively in terms of $f_{2-k,m'}$ with $m'<m$ to cancel the associated negative powers of $q$.  Since $j(z)$, $\Delta(z)^{-1}$, and $E_{k'}(z)$ all have integral coefficients, one sees inductively that the coefficients of the Faber polynomial are all integers.  Denote $\widetilde{F}_m(x):=F_m(x+1728)$, so that $\widetilde{F}_m(x-1728)=F_m(x)$.  Then in particular
$$
\widetilde{F}_m\left(\frac{E_6^2}{\Delta}(z)\right) = F_m(j(z)).
$$
We denote the $r$-th coefficient of the polynomial $\widetilde{F}_m$ by $c_{m,r}$.  Our goal will be to determine the asymptotic growth of the coefficients $c_{m,r}$ in terms of $m$ and $r$.  

In order to establish asymptotics for these coefficients, we will investigate asymptotic growth for derivatives of certain Poincar\'e series.  For an integer $m$ and a function $\varphi_m:\R^+\to \C$ satisfying $\varphi_m(y)=O\left(y^{\alpha}\right)$ for some $\alpha\in \R$ as $y\to 0$, the \begin{it}Poincar\'e series\end{it} $\PP(m,k,\varphi_m;z)$ is defined by 
$$
\PP(m,k,\varphi_m;z):=\sum_{A\in \Gamma_{\infty} \backslash \SL_2(\Z)} \varphi_m^* |_k A(z)
$$
where
$$
\varphi_m^*(z):=\varphi_m(y) e^{2\pi i mx}
$$
and 
$$
f|_k \left(\begin{matrix} a &b\\ c &d\end{matrix}\right)(z) = \left(cz+d\right)^{-k}f\left(\frac{az+b}{cz+d}\right)
$$ 
is the usual weight $k$ slashing operator.  Choosing $\varphi_m(y)=e^{-2\pi my}$ (so that $\varphi_m^*(z) = q^m$) for $k\geq 2$ leads to the classical family of holomorphic Poincar\'e series
$$
P(m,k;z):=\PP(m,k,e(imy);z),
$$
while choosing 
$$
\varphi_m(y):=\mathcal{M}_{\frac{k}{2}}\left(4\pi m y\right)
$$
with 
$$
\mathcal{M}_{s}(y):= |y|^{-\frac{k}{2}} M_{\left(1-\frac{k}{2}\right)\sgn(y),\, s-\frac{1}{2}}\left(|y|\right),
$$
where $M_{\nu,\, \mu}(z)$ is the usual $M$-Whittaker function, leads to the Maass-Poincar\'e series (see for example \cite{Fay})
$$
F(m,2-k;z):= \PP(-m,2-k,\varphi_{-m};z).
$$
The Maass-Poincar\'e series are what are known as \begin{it}harmonic weak Maass forms\end{it} (see \cite{BF}), which satisfy the same modularity as modular forms, but where holomorphicity is replaced by the real analytic smoothness condition that they are annihilated by the \begin{it}weight $(2-k)$-th hyperbolic Laplacian\end{it} 
\begin{equation*}
\Delta_{2-k} := -y^2\left( \frac{\partial^2}{\partial x^2}+
\frac{\partial^2}{\partial y^2}\right) + (2-k)iy\left(
\frac{\partial}{\partial x}+i \frac{\partial}{\partial y}\right).
\end{equation*}
The Maass-Poincar\'e series (in a more general setting where the weight can be a half integer and $\SL_2(\Z)$ may be replaced with a congruence subgroup) have played a prominent role in recent years.  For example, Bringmann and Ono have shown that their coefficients satisfy a duality similar to that given by Zagier \cite{BO2}, they were used to determine exact formulas for Ramanujan's mock theta function $f(q)$ (which is the ``holomorphic part'' of a certain Maass-Poincar\'e series), proving the Andrews--Dragonette conjecture \cite{BO1}, and have been used to give lifts from holomorphic cusp forms to harmonic weak Maass forms \cite{BOPNAS}.

Throughout this paper $m$ will denote a positive integer and $k$ will be taken to be at least $2$.  The bounds on $c_{m,r}$ will be established by first determining the growth of $F^{(r)}(m,2-k;i)$, where for a function $f(z)$ with $z\in \H$, we will abuse notation to denote the partial derivative with respect to $y$ by $f'(z)$ and more generally we will denote the $r$-th derivative of $f$ with respect to $y$ by $f^{(r)}(z):=\frac{\partial^r f}{\partial y^r}(z)$.  Our main result will be to show an asymptotic relationship between $c_{m,r}$ and $\left|F^{(a_r)}(m,2-k;i)\right|$ for some $a_r\in \N$ depending on $r$.
\begin{theorem}\label{thm:coeffgrowth}
Suppose $2<k\in 2\Z$, $m\in \N$, and $0\leq r\leq m$.  Then there exist constants $C_1$ depending only on $k$ and a universal constant $C_2$ such that 
$$
c_{m,r}\sim \frac{1}{C_1C_2^r}\times \begin{cases} 
\frac{\left|F^{(2r)}(m,2-k;i)\right|}{\left(2r\right)!} & \text{if }k\equiv 2\pmod{4}\\
\frac{\left|F^{(2r+1)}(m,2-k;i)\right|}{\left(2r+1\right)!} & \text{if }k\equiv 0\pmod{4}.
\end{cases}
$$
\end{theorem}
In order to obtain an asymptotic for $c_{m,r}$ in terms of simple functions of $m$ and $r$ (elementary functions in either variable when the other is fixed), we next determine an asymptotic for $\left|F^{(a_r)}(m,2-k;i)\right|$ with $a_r=2r$ or $2r+1$.  In order to write down our results, we define 
\begin{equation}\label{eqn:Xdef}
X(r,m):=\frac{1}{2}\left(1+\sqrt{1+\frac{4r-2k+3}{2\pi m}}\right).
\end{equation}
Note that for $r\leq m$ one has $1\leq X(r,m)\leq X(m,m)\approx 1.139652204$ for $m\gg k$.
\begin{theorem}\label{thm:growth}
Suppose that $m$ is sufficiently large and $k\geq 2$.  Then $F(m,2-k;z)$ has at most one root on the line $iy$, $y\in \R$.  This root occurs precisely at $z=i$ if and only if $k\equiv 0\pmod{4}$ and in that case it is a simple root.  Moreover, for $r\leq m$, 
\begin{multline}\label{eqn:growthk2mod4}
\left|F^{(2r)}(m,2-k;i)\right|\sim \Gamma(k)\bigg(1 + C(2r,m)X(2r,m)^{2r-\frac{k-1}{2}+\frac{1}{4}} \\
\cdot\exp\left(-2\pi m(X(2r,m)-1)^2\right) \bigg) \left(2\pi m\right)^{2r} e^{2\pi m}
\end{multline}
in the case that $k\equiv 2\pmod 4$, where 
\begin{equation}\label{eqn:Cdef}
C(r,m) :=\frac{1}{\sqrt{X(r,m)+1-X(r,m)^{-1}}}=O(1),
\end{equation}
while 
\begin{multline}\label{eqn:growthk0mod4}
\left|F^{(2r+1)}(m,2-k;i)\right|\sim \Gamma(k)\bigg(1 + C(2r+1,m)X(2r+1,m)^{2r+1-\frac{k-1}{2}+\frac{1}{4}}\\
\cdot \exp\left(-2\pi m(X(2r+1,m)-1)^2\right) \bigg) (2\pi m)^{2r+1}e^{2\pi m}
\end{multline}
in the case that $k\equiv 0\pmod{4}$.
\end{theorem}
\begin{remarks}
\noindent
\begin{enumerate}
\item
It is worth noting that while $C(r,m)$ is not a constant, the fact that $X(r,m)$ is bounded from above and below by a constant means that $C(r,m)$ also has this property.  Indeed, for $r=o(m)$ one has $C(r,m)\sim 1$.  Moreover, in the case that $r=o\left(\sqrt{m}\right)$, the asymptotic in \eqref{eqn:growthk2mod4} rather pleasantly becomes
$$
\left|F^{(2r)}(m,2-k;i)\right|\sim 2\Gamma(k) \left(2\pi m\right)^{2r} e^{2\pi m},
$$
while, under the same restrictions, \eqref{eqn:growthk0mod4} becomes 
$$
\left|F^{(2r+1)}(m,2-k;i)\right|\sim 2\Gamma(k) \left(2\pi m\right)^{2r+1} e^{2\pi m}.
$$
\item
Duke and Jenkins \cite{DukeJenkins08} have shown that for $m\geq 2d$ all of the zeros of $f_{2-k,m}(z)$ lie on the unit circle.  It is not too difficult to show that $F(m,2-k;z)$ grows asymptotically (in $m$) like $f_{2-k,m}(z)$, and hence it does not come as a surprise that there are no zeros on the line $z=iy$ for sufficiently large $m$.  However, they show explicit examples where there exists a zero outside of the unit circle.  It might be interesting to investigate whether such a zero is ever contained on the line $z=iy$ and whether the condition of $m$ sufficiently large is necessary.
\end{enumerate}
\end{remarks}
Theorem \ref{thm:growth} leads to the following more precise version of Theorem \ref{thm:coeffgrowth} involving the growth of the coefficients $c_{m,r}$ of the Faber polynomial.  To describe our results, we first define the constants (independent of $m$ and $r$)
$$
C_{1}:=\begin{cases} \frac{E_{k'}(i)}{\Delta^{d+1}(i)} & \text{if }k\equiv 2\pmod{4},\\
\frac{\left(E_{k'}\right)'(i)}{\Delta^{d+1}(i)} & \text{if }k\equiv 0\pmod{4},
\end{cases}
$$
and 
$$
C_2:=\frac{\left(E_6'(i)\right)^2}{\Delta(i)}\approx 585.200048.
$$
\begin{remark}
While $C_2$ is also independent of $k$, $C_1$ depends on $k$, but in a very predictable way, since it only depends on
$$
\Delta(i)^{-d} \approx \left( 536.4954009\right)^d
$$
and $k'$, which only depends on $k\pmod{12}$.
\end{remark}
\begin{theorem}\label{thm:coeffvary}
Assume $2<k\in 2\Z$.  Whenever $k\equiv 2\pmod{4}$, for $0\leq r\leq m-2$ one has that
\begin{equation}\label{eqn:Faber2mod4}
c_{m,r}\sim  \left(1 + C(2r,m)X(2r,m)^{2r-\frac{k-1}{2}+\frac{1}{4}}\exp\left(-2\pi m(X(2r,m)-1)^2\right)\right) \frac{\left(2\pi m\right)^{2r} e^{2\pi m}}{(2r)!C_1C_2^r},
\end{equation}
while whenever $k\equiv 0\pmod{4}$, one has 
\begin{multline}\label{eqn:Faber0mod4}
c_{m,r}\sim \bigg(1 + C(2r+1,m)X(2r+1,m)^{2r+1-\frac{k-1}{2}+\frac{1}{4}}\\
\cdot \exp\left(-2\pi m(X(2r+1,m)-1)^2\right) \bigg) \frac{(2\pi m)^{2r+1}e^{2\pi m}}{(2r+1)!C_1C_2^r}.
\end{multline}
\end{theorem}
\begin{remark}
In the case that $r=o\left(\sqrt{m}\right)$ we note again that this becomes
$$
c_{m,r}\sim 2\left(\frac{\left(2\pi m\right)^{2r} e^{2\pi m}}{(2r)!C_1C_2^r}\right)
$$
whenever $k\equiv 2\pmod{4}$ and 
$$
c_{m,r}\sim 2\left(\frac{(2\pi m)^{2r+1}e^{2\pi m}}{(2r+1)!C_1C_2^r}\right)
$$
whenever $k\equiv 0\pmod{4}$.
\end{remark}

The paper is organized as follows.  In Section \ref{section:constant}, we recall the Fourier expansion of the Maass-Poincar\'e series, due to Bringmann and Ono \cite{BOPNAS}, and establish an equality for the coefficient $c_{m,0}$ in terms of a certain linear combination of Maass-Poincar\'e series $F(n,2-k;i)$ or their derivatives, leading naturally to the consideration of derivatives of Poincar\'e series in determining the growth of the coefficients of the polynomial.  In Section \ref{section:derivs}, we prove Theorem \ref{thm:growth}.  Along the way, we prove a lemma which gives an asymptotic for the $n$-th repeated integral (and, more generally, the $\ell$-th repeated integral, where $\ell$ can be taken to be any real number, following the definition given in fractional calculus) of the Gauss error function, which are related to the parabolic cylinder functions (see \cite{Miller}, p. 76) and have been studied going back to Hartree \cite{Hartree} due to their role in physics and chemistry.  In Section \ref{section:coeffs}, we prove Theorem \ref{thm:coeffvary} by showing that the constant $c_{m,r}$ times $(2r)! C_1 C_2^r$ (resp. $(2r+1)! C_1C_2^r$) is asymptotically equal to $F^{(2r)}(m,2-k;i)$ (resp. $F^{(2r+1)}(m,2-k;i)$) whenever $k\equiv 2\pmod{4}$ (resp. $k\equiv 0\pmod{4}$) and then invoking Theorem \ref{thm:growth}.

\section{Evaluating the constant term}\label{section:constant}
The goal of this section will be to determine a formula for the constant term of the Faber polynomial in terms of the derivatives of the Poincar\'e series.
\begin{proposition}\label{prop:constant}
For $2<k\in 2\Z$, there exist constants $b_1,\dots, b_{d}\in \Z$ such that 
\begin{equation}\label{eqn:constgen}
c_{m,0}=\frac{1}{\Gamma(k)} \left(\frac{F(m,2-k;i)}{C_1} - \sum_{n=1}^{d} b_n \frac{F(n,2-k;i)}{C_1}\right)
\end{equation}
In particular, for $k=12$ one has
\begin{equation}\label{eqn:constDelta}
c_{m,0}=\frac{1}{11!}\left(\frac{F'(m,-10;i)\Delta^{2}(i)}{\left(E_{14}\right)' (i)} - \tau(m)\frac{F'(1,-10;i)\Delta^{2}(i)}{\left(E_{14}\right)'(i)}\right).
\end{equation}
\end{proposition}
Since $c_{m,0},b_1,\dots,b_{d}\in \Z$, the following corollary about the rank of the $\Z$-module generated by $F(m,2-k;i)$ when $k\equiv 2\pmod{4}$ and $F_y(m,2-k;i)$ when $k\equiv 0\pmod{4}$ follows immediately.
\begin{corollary}
Suppose that $k\geq 2$.  Then the $\Z$-module generated by
$$
S:=
\begin{cases}
\left\{ F(m,2-k;i)\big| m\in \N\right\}& k\equiv 2\pmod{4},\\
\left\{ F'(m,2-k;i)\big| m\in \N\right\}& k\equiv 0\pmod{4},
\end{cases}
$$
has rank at most $d+1$.  
\end{corollary}

Our argument will go through the Fourier expansion of the Poincar\'e series.  Bringmann and Ono \cite{BOPNAS} have shown that $F(m,2-k;z)$ has the following Fourier expansion.
\begin{proposition}[Bringmann-Ono \cite{BOPNAS}]\label{prop:Fcoeff}
$$
F(m,2-k;z)=(1-k)q^{-m}
\left(
\Gamma(k-1,4 \pi my)-\Gamma(k-1)
\right)
+
\sum_{n \in \Z} c_y(n)q^n.
$$
For  $n\ne 0$
\begin{multline*}
c_y(n) = 2\pi i^{k} \abs{\frac{m}{n}}^{\frac{k-1}{2}}\sum_{c>0} \frac{K_{2-k}(-m,n,c)}{c}\\ 
\times
 \begin{cases}(1-k) \Gamma\left(k-1, \left|4\pi n y\right|\right) J_{k-1}\left(\frac{4\pi}{c} \sqrt{\abs{mn}}\right) & n< 0, \\ 
-\Gamma(k) I_{k-1}\left(\frac{4\pi}{c} \sqrt{\abs{mn}}\right) & n> 0,
\end{cases} 
\end{multline*}
and
$$
c_{y}(0) =  -(2\pi i)^{k} m^{k-1} \sum_{c>0} \frac{K_{2-k}(-m,0,c)}{c^{k}}.
$$
\end{proposition}

We begin with the proof of Proposition \ref{prop:constant}.
\begin{proof}[Proof of Proposition \ref{prop:constant}]
Consider the harmonic weak Maass form 
$$
f_{2-k,m}(z)-F(m,2-k;z).
$$
Let the principal part of $f_{2-k,m}(z)$ be given precisely by $q^{-m} +\sum_{n=1}^{d} b_n q^{-n}$.  Note that since $j(z)$, $\Delta^{-1}(z)$, and $E_{k'}(z)$ all have integer coefficients, all coefficients of $f_{2-k,m}(z)$ are integers, and hence in particular $b_n\in \Z$.  Recall that a harmonic weak Maass form which maps to a cusp form under the operator $\xi_{2-k}:=2i y^{2-k}\overline{\frac{\partial}{\partial \overline{z}}}$ whose principal part is constant must be zero (for example, see Lemma 7.5 of \cite{OnoBook}).  Thus 
\begin{equation}\label{eqn:fmeval}
f_{2-k,m}(z)= \frac{1}{\Gamma(k)} \left( F(m,2-k;z)+\sum_{n=1}^{d} b_n F(n,2-k;z)\right),
\end{equation}
since the difference has no principal part and maps to a cusp form.  We then use the fact that 
\begin{equation}\label{eqn:fmFtilde}
f_{2-k,m}(z) = \widetilde{F}_m\left(\frac{E_6^2(z)}{\Delta(z)}\right)\cdot \left(\frac{E_{k'}(z)}{\Delta^{d+1}(z)}\right)=\frac{E_{k'}(z)}{\Delta^{d+1}(z)}\sum_{n=0}^{m-d-1}c_{m,n} \left(\frac{E_6^2(z)}{\Delta(z)}\right)^n.
\end{equation}
Since $E_6(i)=0$ and $\Delta$ has no roots in the upper half plane, it follows that 
$$
\lim_{z\to i}\left(\frac{E_6^2(z)}{\Delta(z)}\right)^n=0
$$
unless $n=0$.  We then multiply on both sides of equation \eqref{eqn:fmeval} by $\frac{\Delta^{d+1}(z)}{E_{k'}(z)}$ and take the limit $z\to i$, giving the first statement.

In the case $k=12$, by the work of Bringmann and Ono \cite{BOPNAS} we have 
$$
\xi_{2-k}\left(F(m,2-k;z)\right) = (k-1)\left(4\pi m\right)^{k-1}P(m,k;z).
$$
Since the space $S_{12}$ is one dimensional, one has $P(m,k;z)=c_m \Delta$.  One obtains $\frac{\Gamma(11)}{(4\pi m)^{11}}$ times the $m$-th Fourier coefficient of $\Delta$ by integrating against $P(m,k;z)$ (cf. \cite{IwaniecKowalski}, p. 359), so that 
$$
\frac{\Gamma(11)}{(4\pi m)^{11}} \tau(m) = \left<\Delta,P(m,k;z)\right> = c_m\|\Delta\|^2.
$$
Hence $\frac{P(m,k;z)}{P(1,k;z)}=\frac{c_m}{c_1}=\frac{\tau(m)}{m^{11}}$.  Therefore, it follows that  
$$
\xi_{2-k}\left(F(m,2-k;z) - \tau(m)F(1,2-k;z)\right) =0,
$$ 
and hence 
$$
f_{2-k,m}=\frac{1}{11!}\left(F(m,2-k;z) - \tau(m)F(1,2-k;z)\right),
$$
so that $b_1=\tau(m)$.
\end{proof}

\section{Derivatives of Poincar\'e series}\label{section:derivs}
We will first show the asymptotic growth for $F^{(2r+\delta)}\left(m,2-k;i\right)$, where $\delta=0$ if $2-k\equiv 0\pmod{4}$ and $\delta=1$ if $2-k\equiv 2\pmod{4}$. Our argument will be based on the Fourier expansion of the Poincar\'e series.

The following technical lemma will be helpful in establishing Theorem \ref{thm:growth} and includes bounds for the $n$-th iterated integral of the error function whenever $\ell=n$ is taken to be an integer, generalizing work of Gautschi \cite{Gautschi1}, which may be of independent interest within chemistry and physics due to the emergence of these special functions in those fields.
\begin{lemma}\label{lem:intbound}
Let $0\leq A,B\in \R$ and $\ell=\ell(A)\in \R$ be given such that if $\ell<0$ then $\ell$ is a fixed constant with respect to $A$ and $B$ is a fixed constant independent of $A$ and $\ell$.  We denote
$L:=2\ell+1$ and 
$$
X_0:=X_0(\ell,A,B):=\frac{1}{2}\left(1+\sqrt{1+\frac{2L}{A^2B}}\right)
$$
for brevity.
%$\ell=o\left(\left(A\sqrt{B}\right)^{\frac{4}{3}}\right)$.  Then for $B$ a fixed constant,

For $L < BA^2$, one has the asymptotic
\begin{multline}\label{eqn:intbound}
\II:=\II_{\ell,A,B}:=\int_{0}^{\infty} x^{\ell} e^{-B(\sqrt{x}-A)^2}\dx \\
\sim \frac{2\sqrt{\pi}}{\sqrt{B}}\cdot \frac{\left(AX_0\right)^{L}}{\sqrt{1+X_0-X_0^{-1}}} \exp\left(-\frac{A^2B}{4}\left(-1+\sqrt{1+\frac{2L}{A^2B}}\right)^2\right)
\end{multline}
as $A\to\infty$.
\end{lemma}
\begin{proof}
We first shift $x\to \left(x+A\right)^2$ to rewrite 
\begin{equation}\label{eqn:intrewrite}
\II = 2\int_{-A}^{\infty} (x+A)^{2\ell+1} e^{-Bx^2}\dx.
\end{equation}
Up to a normalization, this is the $(2\ell+1)$-th integral of the Gauss error function
\begin{equation}\label{eqn:erfdef}
\erf(y):=\frac{2}{\sqrt{\pi}} \int_0^y e^{-x^2}\dx
\end{equation}
evaluated at $-A$.  Due to the appearance of these integrals in chemistry and physics, asymptotics have been extensively studied when $2\ell+1\in \N$.  Asymptotics in the case $-A<0$ were given by Gautschi \cite{Gautschi1} when $\ell=O(A)$. 

First assume that $\ell=O\left(A^2B\right)$.  We next pull $A^{L}$ out of the integral and then make the change of variables $x\to \frac{x}{\sqrt{B}} +\frac{a}{2AB}$ for 
\begin{equation}\label{eqn:adef}
a:=A^2B\left(-1+\sqrt{1+\frac{2L}{A^2B}}\right).
\end{equation} 
This gives
\begin{equation}\label{eqn:IIshift}
\II=2\left(\frac{A^{L}}{\sqrt{B}}\right) \int_{-A\sqrt{B}-\frac{a}{2AB}}^{\infty} \left(1+ \frac{ x +\frac{a}{2A\sqrt{B}}}{A\sqrt{B}}\right)^{L} \exp\left( - \left(x+\frac{a}{2A\sqrt{B}}\right)^2\right)\dx.
\end{equation}
We now use the equation
\begin{equation}\label{eqn:logasym}
\left(1+\frac{1}{f(A)}\right)^{g(A)} = \exp\left(g(A)\ln\left(1+f(A)^{-1}\right)\right)\sim \exp\left(\sum_{n=1}^{\infty}(-1)^{n+1} \frac{g(A)}{n f(A)^n}\right),
\end{equation}
valid whenever $f(A)\geq 1$ for $A$ sufficiently large, with $g(A)=L$ and $f(A)=\frac{A\sqrt{B}}{x+\frac{a}{2A\sqrt{B}}}$.  Here the condition for $f(A)$ is satisfied because $L< A^2B$ and hence 
$$
\frac{a}{A^2B}=-1+\sqrt{1+\frac{2L}{A^2B}}< \sqrt{3}-1<1.
$$
We now expand $f(A)^{-n}$ using the binomial theorem.  The $n$-th term of the sum becomes
$$
\frac{(-1)^{n+1}}{n} \sum_{j=0}^{n}\binom{n}{j} \frac{x^{n-j}La^{j}}{2^j\left(A\sqrt{B}\right)^{n+j}}.
$$
When $a=o\left(A\sqrt{B}\right)$, then this sum is clearly asymptoticaly $o(1)$ for all $n\geq 1$.  Otherwise the asymptotic is increasing as a function of $j$, and for $j< n-2$ the terms are bounded by 
$$
O\left( \frac{a^{n-3}L}{\left(A\sqrt{B}\right)^{2n-3}}\right).
$$
Since $a=O(A^2B)$ and $L=O(A^2B)$, this becomes 
$$
O\left(\left(A\sqrt{B}\right)^{-1}\right)=o(1).
$$
Therefore, setting $Y:=\frac{a}{2A^2B}$, the exponential in \eqref{eqn:logasym} is asymptotically equal to 
\begin{multline}\label{eqn:dominate1}
\exp\left[\left(\frac{L}{a} \sum_{n=2}^{\infty} \frac{(-1)^{n-1}}{n}2\binom{n}{2} Y^{n-1}\right)x^2 -\left(\frac{2LA\sqrt{B}}{a} \sum_{n=1}^{\infty}\left(-Y\right)^{n}\right)x \right.\\
\left.+ L\sum_{n=1}^{\infty} \frac{(-1)^{n+1}}{n} Y^n\right]\\
=\exp\left(\frac{L}{a}\left(\frac{1}{(1+Y)^2}-1\right)x^2 -\frac{2L}{a}\left(A\sqrt{B}\right) \left(\frac{1}{(1+Y)}-1\right) x + L\ln(1+Y)\right),
\end{multline}
since $\frac{2\binom{n}{2}}{n}=n-1$ and the first sum is then merely the power series expansion of the derivative of the geometric series with the $n=1$ term missing.  We now look at the coefficient in front of $x$ in the integrand of \eqref{eqn:IIshift}.  This equals 
\begin{multline*}
\frac{2L}{a}\left(A\sqrt{B}\right) \left(\frac{Y}{1+Y}\right) - \frac{a}{A\sqrt{B}}\\
=\frac{A\sqrt{B}}{a} \left(2L\left(\frac{-1+\sqrt{1+\frac{2L}{A^2B}}}{1+\sqrt{1+\frac{2L}{A^2B}}}\right) - A^2B\left( -1+\sqrt{1+\frac{2L}{A^2B}}\right)^2\right)\\
= \frac{A\sqrt{B}}{a} \left(2L\left(\frac{\left(-1+\sqrt{1+\frac{2L}{A^2B}}\right)^2}{-1+\left(1+\frac{2L}{A^2B}\right)}\right) - A^2B\left( -1+\sqrt{1+\frac{2L}{A^2B}}\right)^2\right)=0.
\end{multline*}
We now determine the coefficient in front of $x^2$.  This equals
$$
\frac{L}{a}\left(\frac{1}{(1+Y)^2}-1\right) -1.
$$
Using the fact that $(1+Y)a=L$ and $X_0=1+Y$, this then equals
$$
X_0^{-1}-X_0-1 = -\left(1+X_0-X_0^{-1}\right).
$$
Noting that $1\ll X_0\ll 1$, we may consider the statement of the lemma for convergent subsequences where the limit $\lim_{A\to\infty} X_0(\ell,A,B)$ exists.  Hence the coefficient of $x^2$ converges to a fixed value, and the fact that $\ell<0$ implies that $\ell$ is a fixed constant shows that the coefficient of $x^2$ converges to a constant less than or equal to $-1$.  By first pulling the terms 
\begin{equation}\label{eqn:constterms}
\left(1+Y\right)^L\exp\left(-\frac{a^2}{4A^2B}\right)
\end{equation}
from the constant coefficients of the integrand, we can hence use the dominated convergence theorem (considering the integral over the entire real line where the function is zero outside of the support) and the value of the error function over the whole real line to conclude that 
$$
\II\sim \frac{2A^{L}}{\sqrt{B}} X_0^L\exp\left(-\frac{a^2}{4A^2B}\right)\cdot \sqrt{\frac{\pi}{1+X_0-X_0^{-1}}} .
$$
After plugging in the definition \eqref{eqn:adef} of $a$, one sees that this is precisely equation \eqref{eqn:intbound}.

\end{proof}

\begin{proof}[Proof of Theorem \ref{thm:growth}]
We will separate into the cases where $k\equiv 2\pmod{4}$ and $k\equiv 0\pmod{4}$.  We will only show the $k\equiv 2\pmod{4}$ case here, but the $k\equiv 0\pmod{4}$ case is entirely analogous.  In this case, we plug $z=iy$ into the expansion given for the Fourier coefficients in Proposition \ref{prop:Fcoeff}.  We begin with the expansion given in Proposition \ref{prop:Fcoeff} and directly differentiate $2r$ times with respect to $y$.

For every $n>1$ we separate the $c=1$ term from the sum given in $c_y(n)$ and note that $K_{2-k}(-m,n,1)=1$ to obtain
\begin{multline*}
F^{(2r)}(m,2-k;iy)= \Gamma(k)\Bigg(\left(2\pi m\right)^{2r}e^{2\pi my}\\
 + \sum_{n>0} 2\pi \left(2\pi n\right)^{2r} e^{-2\pi ny} \abs{\frac{m}{n}}^{\frac{k-1}{2}}I_{k-1}\left(4\pi\sqrt{mn}\right)\Bigg) + E_{2r}(y),
\end{multline*}
where 
\begin{multline}\label{eqn:E0def}
E_0(y) := (2\pi)^{k} m^{k-1} \sum_{c>0} \frac{K_{2-k}(-m,0,c)}{c^{k}} + (1-k)\Gamma(k-1,4\pi my)e^{2\pi my}\\
 + 2\pi (k-1) \sum_{n<0} \abs{\frac{m}{n}}^{\frac{k-1}{2}}e^{-2\pi ny}\Gamma(k-1,-4\pi ny)\sum_{c>0} \frac{K_{2-k}(-m,n,c)}{c} J_{k-1}\left(\frac{4\pi}{c} \sqrt{\abs{mn}}\right)\\
+ 2\pi \Gamma(k)\sum_{n>0} \abs{\frac{m}{n}}^{\frac{k-1}{2}}e^{-2\pi ny}\sum_{c>1} \frac{K_{2-k}(-m,n,c)}{c} I_{k-1}\left(\frac{4\pi}{c} \sqrt{mn}\right)
\end{multline}
denotes the sum of all of the terms corresponding to $n\leq 0$ and all of the terms with $n>0$ and $c>1$, and furthermore $E_r(y):=E_0^{(r)}(y)$.

To determine the asymptotic of the derivatives at $z=i$ we plug in $y=1$ and bound $E_r(1)$.  We will first show the asymptotic growth of the main terms 
\begin{equation}\label{eqn:c1terms}
(2\pi m)^{2r}e^{2\pi m} + 2\pi \sum_{n>0} (2\pi n)^{2r} e^{-2\pi n} \abs{\frac{m}{n}}^{\frac{k-1}{2}}I_{k-1}\left(4\pi\sqrt{mn}\right),
\end{equation}
all of which are real and positive.  The term $(2\pi m)^{2r}e^{2\pi m}$ clearly exhibits the growth given in Theorem \ref{thm:growth} with constant $1$, and hence to show that the main terms satisfy the given asymptotic, it suffices to show that 
\begin{equation}\label{eqn:maintoshow}
2\pi \sum_{n>0} (2\pi n)^{2r} e^{-2\pi n} \abs{\frac{m}{n}}^{\frac{k-1}{2}}I_{k-1}\left(4\pi\sqrt{mn}\right)\sim (2\pi m)^{2r}e^{2\pi m}.
\end{equation}
Since $m$ is large, we may use the asymptotic 
\begin{equation}\label{eqn:Ibig}
I_{\alpha}(x) \sim \frac{e^x}{\sqrt{2\pi x}}
\end{equation}
to bound the $I$-Bessel function in each case.  This shows that the terms in \eqref{eqn:c1terms} are asymptotically equal to 
\begin{equation}\label{eqn:c1termsexp}
(2\pi m)^{2r}e^{2\pi m} + \frac{1}{\sqrt{2}} \sum_{n>0} \frac{\left(2\pi n\right)^{2r}}{(mn)^{\frac{1}{4}}}  \abs{\frac{m}{n}}^{\frac{k-1}{2}} e^{-2\pi n+4\pi \sqrt{mn}}.
\end{equation}
Set $\ell:=2r-\frac{k-1}{2}-\frac{1}{4}$.  Denote the $n$-th term of the sum in \eqref{eqn:c1termsexp} by $a_n (2\pi)^{2r} m^{\frac{k-1}{2}-\frac{1}{4}}$.  Consider the function 
\begin{equation}\label{eqn:fdef}
f(x):=\exp\left(\ell\ln(x) - 2\pi x + 4\pi \sqrt{m}\sqrt{x}\right).
\end{equation}
Set 
$$
x_0:=\begin{cases} \left(\frac{1}{2}\sqrt{m} + \frac{1}{2}\sqrt{m+\frac{2}{\pi}\ell}\right)^2 &\text{if }m+\frac{2}{\pi}\ell\sqrt{m}\geq 0,\\
1&\text{otherwise}.
\end{cases}
$$
One easily determines that the function $f(x)$ is increasing as a function of $x$ for $1<x<x_0$ and decreasing for $x\geq  x_0$.  We write 
$$
f_1(x):=\begin{cases} f(1) & x<1, \\ f(x) & 1\leq x \leq x_0, \\ f\left(x_0\right) & x>x_0,\end{cases}
$$
and 
$$
f_2(x):=\begin{cases}f\left(x_0\right) & x<x_0, \\ f(x) & x\geq x_0,\end{cases}
$$
and see clearly that $f_1$ and $f_2$ are monotonic, with $f_1(n)=a_n$ for $1\neq n \leq x_0$ and $f_2(n) = a_n$ for $n>x_0$.  One then bounds the sum 
$$
\sum_{n=1}^{\lfloor x_0\rfloor} a_n \leq \int_{1}^{\lfloor x_0\rfloor+1}f_1(x)\dx =\left(\lfloor x_0\rfloor +1 -x_0\right)f\left(x_0\right) + \int_{1}^{x_0}f(x)\dx  
$$
since the left hand side is a Riemann lower bound for the integral and the integral from $x_0$ to $\lfloor x_0\rfloor+1$ is easily computed.  Meanwhile,
$$
\sum_{n=1}^{\lfloor x_0\rfloor} a_n \geq \int_{0}^{\lfloor x_0\rfloor} f_1(x)\dx = f(1) + \int_{1}^{\lfloor x_0\rfloor } f(x)\dx 
$$
since the sum is a Riemann upper bound for this integral.  Similarly, using the function $f_2(x)$, we obtain the bound
$$
\int_{\lfloor x_0\rfloor +1}^{\infty} f(x)\dx \leq \sum_{n=\lfloor x_0\rfloor +1}^{\infty} f_2(n)\leq \left( x_0 - \lfloor x_0\rfloor\right)f\left(x_0\right)+\int_{x_0}^{\infty} f(x)\dx .
$$
Hence we obtain
$$
f(1)-f(x_0) + \int_{1}^{\infty} f(x)\dx \leq \sum_{n>0} a_n \leq f\left(x_0\right)+\int_{1}^{\infty}f(x)\dx.
$$
We will see later that $f(1)$ and $f\left(x_0\right)$ contribute to the error.  First we will give an asymptotic for the integral
\begin{equation}\label{eqn:c1integral}
\frac{\left(2\pi\right)^{2r}}{\sqrt{2}}m^{\frac{k-1}{2}-\frac{1}{4}} \int_{1}^{\infty} x^{\ell} e^{-2\pi x +4\pi \sqrt{m}\sqrt{x}} \dx.
\end{equation}
We rewrite the integral in \eqref{eqn:c1integral} as
$$
 e^{2\pi m} \int_{1}^{\infty} x^{\ell} e^{-2\pi \left (\sqrt{x}-\sqrt{m}\right)^2} \dx.
$$ 
and then use Lemma \ref{lem:intbound} with $B=2\pi$, $A=\sqrt{m}$ and $\ell = 2r-\frac{k-1}{2} - \frac{1}{4}$ to give the asymptotic 
$$
\sqrt{2}\frac{x_0^{\ell+\frac{1}{2}}}{\sqrt{1+\left(\frac{x_0}{m}\right)^{\frac{1}{2}} - \left(\frac{m}{x_0}\right)^{\frac{1}{2}} }} \exp\left(-2\pi x_0+2\pi m \sqrt{1+\frac{4\ell+2}{A^2B}}\right)
$$
for the integral.  Plugging this into \eqref{eqn:c1integral} and noting that 
$$
2\pi m \sqrt{1+\frac{4\ell+2}{A^2B}}+2\pi m = 2\pi \sqrt{m}\sqrt{x_0}
$$
gives the asymptotic 
\begin{equation}\label{eqn:c1x0bound}
\left(2\pi\right)^{2r}m^{\frac{k-1}{2}-\frac{1}{4}} \int_{1}^{\infty} f(x)\dx\sim  \frac{\left(2\pi\right)^{2r}m^{\frac{k-1}{2}-\frac{1}{4}}x_0^{\ell+\frac{1}{2}}}{\sqrt{1+\left(\frac{x_0}{m}\right)^{\frac{1}{2}} - \left(\frac{m}{x_0}\right)^{\frac{1}{2}} }} \exp\left(-2\pi x_0+2\pi \sqrt{m}\sqrt{x_0}\right).
\end{equation}
We now recall the definition \eqref{eqn:Xdef} of $X(r,m)$ in order to rewrite this as 
\begin{multline}\label{eqn:c1bound}
\frac{\left(2\pi\right)^{2r} e^{2\pi m}m^{\frac{k-1}{2}-\frac{1}{4}}\cdot m^{2r -\frac{k-1}{2}+\frac{1}{4}}X(2r,m)^{2r-\frac{k-1}{2}+\frac{1}{4}}}{\sqrt{ 1+X(2r,m)-X(2r,m)^{-1}}}\exp\left(-2\pi m \left(X(2r,m)-1\right)^2\right)
\\
=\left(2\pi m\right)^{2r} e^{2\pi m}\frac{X(2r,m)^{2r-\frac{k-1}{2}+\frac{1}{4}}}{\sqrt{ 1+X(2r,m)-X(2r,m)^{-1}}}\exp\left(-2\pi m \left(X(2r,m)-1\right)^2\right),
\end{multline}
as desired.  We now return to the terms $f(1)$ and $f\left(x_0\right)$.  The term $f(1)=e^{-2\pi + 4\pi \sqrt{m}}$ clearly is an error term when compared against \eqref{eqn:c1bound}.  When $x_0\neq 1$ we then evaluate the term 
\begin{equation}\label{eqn:fx0}
f\left(x_0\right) = x_0^{\ell} \exp\left(-2\pi x_0 +4\pi \sqrt{m}\sqrt{x_0}\right).
\end{equation}
Comparing with \eqref{eqn:c1x0bound}, we see that 
$$
f\left(x_0\right)=O\left( \frac{\sqrt{1+\left(\frac{x_0}{m}\right)^{\frac{1}{2}} - \left(\frac{m}{x_0}\right)^{\frac{1}{2}}}}{\sqrt{x_0}}\int_{1}^{\infty} x^{\ell} e^{-2\pi \left (\sqrt{x}-\sqrt{m}\right)^2} \dx\right).
$$
Since $m\ll x_0\ll m$, one has that
$$
\sqrt{1+\left(\frac{x_0}{m}\right)^{\frac{1}{2}} - \left(\frac{m}{x_0}\right)^{\frac{1}{2}}}=O(1),
$$ 
while $\sqrt{x_0}\to\infty$ as $m\to\infty$.  It follows that $f\left(x_0\right)$ contributes to the error.

It remains to show that $E_{2r}(1)$ contributes to the error when compared to \eqref{eqn:c1bound} as well.  We show the case for $r=0$ first.  The constant term clearly exhibits polynomial growth in the variable $m$.  We use the asympotic for the incomplete Gamma function 
\begin{equation}\label{eqn:Gammabig}
\Gamma(s,x) \sim x^{s-1}e^{-x}
\end{equation}
as $x\to \infty$.  Hence 
$$
(1-k)\Gamma(k-1,4\pi m) e^{2\pi m}\ll m^{k-2}e^{-2\pi m}\to 0
$$
as $m\to \infty$.  We next move to bounding the sum of the remaining terms in \eqref{eqn:E0def} containing an incomplete Gamma function.  For $c\ll \sqrt{mn}$ we use the asymptotic for $J_{k-1}(x)$ with $x$ small, namely
\begin{equation}\label{eqn:Jsmall}
J_{\alpha}(x) \sim \frac{1}{\Gamma(\alpha+1)}\left(\frac{x}{2}\right)^{\alpha},
\end{equation}
to obtain 
$$
J_{k-1}\left(\frac{4\pi}{c}\sqrt{\abs{mn}}\right)\ll \frac{\sqrt{c}}{\left(mn\right)^{\frac{1}{4}}}=O(1),
$$
while for $c\gg \sqrt{mn}$ we use the asymptotic 
\begin{equation}\label{eqn:Jbig}
J_{\alpha}(x)\sim \sqrt{\frac{2}{\pi x}} \cos\left(x-\frac{\pi}{2}\alpha -\frac{\pi}{4}\right) 
\end{equation}
for large parameters, giving
$$
J_{k-1}\left(\frac{4\pi}{c}\sqrt{\abs{mn}}\right)\ll \left(\frac{\sqrt{\abs{mn}}}{c}\right)^{k-1}
$$
Bounding the Kloosterman sum trivially by $c$, we now have absolute convergence on the sum in $c>0$ from the factor $c^{k}$ in the denominator.  This gives
\begin{multline}\label{eqn:Gammabound}
\abs{\sum_{n<0}\abs{\frac{m}{n}}^{\frac{k-1}{2}} e^{-2\pi n} \Gamma(k-1,-4\pi n)\sum_{c>0} \frac{K_{2-k}(-m,n,c)}{c} J_{k-1}\left(\frac{4\pi}{c}\sqrt{\abs{mn}}\right)}\\
\ll \sum_{n<0}\abs{\frac{m}{n}}^{\frac{k-1}{2}} \sqrt{\abs{mn}} e^{-2\pi n} \abs{\Gamma(k-1,-4\pi n)}\ll m^{\frac{k}{2}},
\end{multline}
where we have used the asymptotic \eqref{eqn:Gammabig} for the incomplete Gamma function to obtain absolute convergence on the sum in $n<0$.  In the case when $r>0$, we now note that taking derivatives of the incomplete Gamma function changes the asymptotic behaviour by $(2\pi n)^{\alpha}$ for some $\alpha\in \N$ with $\alpha<r\leq m$, while taking the derivative with respect to $y$ of $q^n$ behaves in the same manner.  Hence the exponential decay of the terms shown above will follow through to show absolute convergence in the same way.  Therefore this will contribute to the error term for all $r\in \N_0$. 

Hence only the terms with $c>1$ and $n>0$ remain to bound $E_{2r}(1)$.  We again show the result for $r=0$ and note that the full result follows by multiplying by an appropriate power of $n$.  In these terms we bound the $I$-Bessel function with the asymptotic \eqref{eqn:Ibig} for $x$ large and 
\begin{equation}\label{eqn:Ismall}
I_{\alpha}(x)\sim \frac{1}{\Gamma(\alpha+1)}\left(\frac{x}{2}\right)^{\alpha}
\end{equation}
for $x$ small in order to obtain
\begin{multline*}
\abs{\sum_{c>1}\frac{K_{2-k}(-m,n,c)}{c} I_{k-1}\left(\frac{4\pi}{c}\sqrt{\abs{mn}}\right)}\\
\ll \frac{1}{(mn)^{\frac{1}{4}}} \sum_{1<c\ll m}\abs{\frac{K_{2-k}(-m,n,c)}{\sqrt{c}}} e^{2\pi \sqrt{mn}} + \sum_{c\gg m} \abs{\frac{K_{2-k}(-m,n,c)}{c^{k}}} \sqrt{mn}^{k-1}\\
\ll m^{\frac{5}{4}}n^{-\frac{1}{4}} e^{2\pi \sqrt{mn}},
\end{multline*}
since the second sum converges absolutely and exhibits only polynomial growth in $\sqrt{mn}$.  

Hence the sum of the remaining terms becomes 
\begin{equation}\label{eqn:cgt1terms}
\ll \sum_{n>0} \abs{\frac{m}{n}}^{\frac{k-1}{2}} m^{\frac{5}{4}}n^{-\frac{1}{4}} e^{- 2\pi n+2\pi \sqrt{mn}}\ll m^{\frac{k}{2}+\frac{3}{4}} \sum_{n>0} \frac{1}{n^{\frac{k}{2}-\frac{1}{4}}} e^{- 2\pi n +2\pi \sqrt{mn}}
\end{equation}
In the range $n\gg m^{1+\varepsilon}$, we have $e^{- 2\pi n+2\pi \sqrt{mn}}\ll e^{(-2\pi+\varepsilon)n}$, giving absolute convergence in this range, while the maximal value for the exponential in the range $n\ll m^{1+\varepsilon}$ is $e^{\frac{\pi}{2}m}$, which is obtained at $n=\frac{m}{4}$.  Hence the sum is bounded by 
\begin{equation}\label{eqn:cgt1bound}
m^{\frac{k}{2}+\frac{7}{4}+\varepsilon} e^{\frac{\pi}{2} m}=o\left(e^{2\pi m}\right).
\end{equation}
Thus we have established that $E_0(1)=o\left(e^{2\pi m}\right)$.

We now move on to bounding $E_{2r}(1)$.  The term with $n=0$ disappears for $r>0$.  For the terms with $n>0$ and $c>1$, we note that each of the terms in \eqref{eqn:cgt1terms} is multiplied by $(2\pi n)^{2r}$ for the corresponding term in $E_{2r}(1)$.  For $r\leq \frac{k}{2}+1$ this simply multiplies the bound in \eqref{eqn:cgt1bound} by $(2\pi m)^{2r}$, while for $r>\frac{k}{2}$ we bound the sum in \eqref{eqn:cgt1terms} by the corresponding integral and then complete the square, which gives a term $e^{\frac{\pi}{2}m}$ while rewriting the integral as one from Lemma \ref{lem:intbound} with $B=2\pi$, $A=\frac{\sqrt{m}}{2}$, and $\ell=2r-\frac{k-1}{2}-\frac{1}{4}$.  Lemma \ref{lem:intbound} then shows that this sum is bounded from above by 
$$
m^{\frac{k}{2}+\frac{3}{4}}(2\pi m)^{2r} e^{\frac{\pi}{2}m}=o\left((2\pi m)^{2r} e^{2\pi m}\right)
$$
and this term hence still contributes to the error.

It remains to bound the terms with $n<0$ from $E_{2r}(1)$.  We first evaluate the derivative of $\Gamma(k-1,4\pi n y) e^{2\pi n y}$.  Using the product rule, if we always take the derivative of $e^{2\pi ny}$ and evaluate at $y=1$, then this gives 
$$
(2\pi n)^r \Gamma(k-1,4\pi n) e^{2\pi n},
$$
while otherwise we took the derivative of $e^{2\pi ny}$ the first $j$ times and then took the derivative of $\Gamma(2-k,4\pi ny)$.  After this, we have 
\begin{equation}\label{eqn:Gammaderiv1}
-(2\pi n)^{j} (4\pi n)^{k-1} y^{k-2} e^{-2\pi ny}.
\end{equation}
Taking the derivative of \eqref{eqn:Gammaderiv1} $2r$ times, we keep track of how many times we have taken the derivative of $y^{k-2}$.  With this accounting, the derivative evaluated at $y=1$ becomes
\begin{multline}\label{eqn:Gammaderiv2}
(2\pi n)^{2r} \Gamma(k-1,4\pi n) e^{2\pi n}\\ -\left(4\pi n\right)^{k-1}e^{-2\pi n} \sum_{i=1}^{k-1}\sum_{j=0}^{2r-i} (2\pi n)^j \binom{r-j-1}{i-1}(k-2)_{i-1} (-2\pi n)^{2r-j-i}.
\end{multline}
We simplify so that the sum in \eqref{eqn:Gammaderiv2} becomes
\begin{multline}\label{eqn:Gammaderiv3}
\sum_{i=1}^{k-1}(-1)^i (2\pi n)^{2r-i}(k-2)_{i-1} \sum_{j=0}^{2r-i} (-1)^j \binom{2r-j-1}{i-1}\\
 = \sum_{i=1}^{k-1}(-1)^i (2\pi n)^{2r-i}(k-2)_{i-1} \sum_{j=0}^{2r-i} (-1)^j \binom{r-j-1}{2r-i}\ll \sum_{i=1}^{k-1} (2\pi n)^{2r-i} (2\pi m)^{i},
\end{multline}
by bounding the binomial coefficient naively against $(2r)^{i-1}$ and using $r\leq m$.  We now bound the incomplete Gamma function with \eqref{eqn:Gammabig}, so that both terms are of the same asymptotic size and the corresponding sums may be treated simultaneously.  

Noting that the maximal value of $(2\pi n)^{\ell} e^{-2\pi n}$ occurs at $2\pi n=r$, the maximal value from the sum $\sum_{n\geq 1} (2\pi n)^{\ell} e^{-2\pi n}$ contributes to the error, and we may bound against the integral as we did in the main case.  In this case, bounding by the integral $\int_{\ell}^{\infty} x^{\lceil \ell\rceil}e^{-x}$ and using integration by parts $\lceil \ell \rceil$ times gives the bound $\ell^{\lceil\ell\rceil+1}e^{-\ell}\ll m^{\lceil\ell\rceil+1}$, exhibiting only polynomial growth in $m$.  This concludes the proof of equation \eqref{eqn:growthk2mod4}.

We now show the statement that there is at most one root of $F(m,2-k;z)$ on the line $iy$.  Note that for $y\neq 1$ and $r=0$, the terms in equation \eqref{eqn:c1termsexp} are replaced by 
\begin{equation}\label{eqn:c1termsexpy}
e^{2\pi my} - i^k\sqrt{\pi} \sum_{n>0} \frac{1}{(mn)^{\frac{1}{4}}}  \abs{\frac{m}{n}}^{\frac{k-1}{2}} \exp\left(-2\pi y \left(\sqrt{n}-\frac{\sqrt{m}}{y}\right)^2+2\pi \frac{m}{y}\right)
\end{equation}
When $y<1$, since all terms are positive, the sum is bounded from below by 
$$
\sum_{n=\frac{m}{y}}^{\frac{m}{y}+\sqrt{m}}\sum_{n>0} \frac{1}{(mn)^{\frac{1}{4}}}  \abs{\frac{m}{n}}^{\frac{k-1}{2}} \exp\left(-2\pi y \left(\sqrt{n}-\frac{\sqrt{m}}{y}\right)^2+2\pi \frac{m}{y}\right) \gg \exp\left(2\pi\frac{\sqrt{m}}{y}\right),
$$
and hence dominates the term $e^{2\pi my}$.  Thus $iy$ cannot be a root of $F(m,2-k;z)$ as this sum exhibits exponential growth and the terms $E_0(y)$ will still contribute to the error.

For $y>1$, one similarly shows that $e^{2\pi my}$ dominates the terms of the sum, and hence $iy$ also cannot be a root for $m$ sufficiently large.

In the case $k\equiv 2\pmod{4}$ there is no such root, while for $k\equiv 0\pmod{4}$ there is always a root by modularity.  Since the first derivative at $z=i$ grows asymptotically as $e^{2\pi m}$ in this case, we know that for $m$ sufficiently large the root must be simple.
\end{proof}
We will also need the following simpler bound whenever $m$ is fixed and the number of derivatives is taken to go to $\infty$.  Denote the \begin{it}holomorphic part\end{it} of $F(m,2-k;z)$ by 
$$
F(m,2-k;z)^+ = \Gamma(k) q^{-m} + \sum_{n\geq 0} c_y(n) q^n,
$$
with $c_y(n)$ given in Proposition \ref{prop:Fcoeff}, and likewise denote the $r$-th derivative with respect to $y$ by $F^{(r)}(m,2-k;z)^+$. 
\begin{proposition}\label{prop:mfixedgrowth}
When $m$ is fixed while $r\to \infty$ we have the bound
\begin{equation}\label{eqn:mfixed}
F^{(r)}(m,2-k;i)^+=O\left(\left(\frac{\ell+\frac{1}{2}}{2\pi e^{1-\varepsilon}}\right)^{\ell+\frac{1}{2}}\right)=O\left( \left(\frac{1}{2\pi} + \varepsilon\right)^{\ell} \Gamma\left(\ell+\frac{1}{2}\right)\ell^{\frac{1}{2}}\right),
\end{equation}
where $\ell=r-\frac{k-1}{2}-\frac{1}{4}$.  In particular, when $k=2$ we have 
\begin{equation}\label{eqn:mfixedk0}
F^{(r)}(m,0;i)=F^{(r)}(m,0;i)^+\ll_{\varepsilon} \left(\frac{1}{2\pi} + \varepsilon\right)^{r} \Gamma\left(r-\frac{1}{4}\right)r^{\frac{1}{2}},
\end{equation}
\end{proposition}
\begin{proof}
First we see that for $m$ fixed and $r\to \infty$, the term 
$$
(2\pi m)^{2r}e^{2\pi m}=O\left(c^r\right)
$$
for some constant $c$.

We now deal with the terms coming from $c_y(n)$ with $c=1$.  Since $m\leq r$ we have 
$$
(2\pi m)^{2r}e^{2\pi m}=O\left(\left(2\pi e^{\frac{m}{r}\pi}m\right)^{2r}\right).
$$
One also sees that when $\ell\to \infty$ the maximum occurring in \eqref{eqn:c1termsexp} occurs at $n$ equal to 
$$
x_0=\left(\frac{1}{2}\sqrt{m}+\frac{1}{2}\sqrt{m+\frac{2\ell}{\pi}}\right)^2.
$$
But then the maximal value from the sum \eqref{eqn:c1termsexp} is
$$
f\left(x_0\right) = \left(\frac{x_0}{\exp\left(1-\frac{m}{\ell}\pi - \pi \sqrt{\left(\frac{m}{\ell}\right)^2+\frac{2}{\pi}\frac{m}{\ell}}\right)}\right)^{\ell},
$$
where $f$ is the function defined in \eqref{eqn:fdef}.  Since $\frac{m}{\ell}\to 0$, this gives the estimate
$$
f\left(x_0\right)\ll_{\varepsilon} \left(\frac{x_0}{e^{1-\varepsilon}}\right)^{\ell}.
$$
We then write 
$$
x_0=\ell\left(\frac{1}{2}\sqrt{\frac{m}{\ell}} + \frac{1}{2}\sqrt{\frac{m}{\ell} + \frac{2}{\pi}}\right)^2\ll\frac{\ell}{2\pi}\left(1+\varepsilon\right)\ll \frac{\ell}{2\pi} e^{\varepsilon}.
$$
Obviously $f(1)=O\left(c^r\right)$, so for the terms not contained in $E_{r}(1)$ it remains to show that the integral contributes to the error in this case.  For this, consider the integral in \eqref{eqn:intrewrite} with $A=\sqrt{m}$, $B=2\pi$ and $\ell$ as chosen above.  

We set 
$$
a_0:=\frac{A}{2}\left(-1+\sqrt{ 1+2\left(\frac{2\ell+1}{A^2B}\right)}\right)
$$
so that the maximum of the value inside the integral 
\begin{equation}\label{eqn:geqn}
\int_{-A}^{\infty} (x+A)^{2\ell +1} e^{-Bx^2}\dx
\end{equation}
occurs at $x=a_0$.  Call the integrand $g(x)$.  We write $x=a_0+y$ so that the integral is given by
$$
\int_{-a_0}^{\infty} \exp\left( \left(2\ell+1\right)\ln\left( a_0+A+ x\right) -B\left(a_0+x\right)^2 \right)\dx 
$$
We expand the exponential as
$$
\left(-Ba_0^2\ + \left(2\ell+1\right) \ln\left(a_0+A\right)\right) -B a_0 x -Bx^2 + \left(2\ell+1\right)\ln\left(1 + \frac{x}{a_0+A}\right).
$$
The first two grouped terms give the maximal value $g\left(a_0\right)$, while the last term can be bounded by
$$
\left(1+\frac{x}{a_0+A}\right)^{2\ell+1}\ll e^{\frac{2\ell+1}{a_0+A} x}.
$$
This gives the bound for the integral \eqref{eqn:geqn} of 
\begin{equation}\label{eqn:largebound}
g\left(a_0\right) \int_{-a_0}^{\infty} e^{-Ba_0 x -Bx^2 + \frac{2\ell+1}{a_0+A} x}\dx\ll g\left(a_0\right).
\end{equation}
It remains to bound $g\left(a_0\right)$.  Bounding 
$$
1+\sqrt{1+2\left(\frac{2\ell+1}{A^2B}\right)} \ll e^{\varepsilon} \sqrt{2\left(\frac{2\ell+1}{A^2B}\right)}
$$ 
and denoting $2\ell+1=L$, the fact that $A$ and $B$ are constants implies
$$
g\left(a_0\right) \ll \left(\abs{\frac{A}{2}}^{L}\frac{2^{\frac{L}{2}} L^{\frac{L}{2}}} {A^{L}B^{\frac{L}{2}}e^{\frac{L}{2}}}\right) e^{\varepsilon \ell},
$$
since
$$
\exp\left(\frac{A^2B}{2}\sqrt{1+\frac{2L}{A^2B}} - \frac{A^2B}{2}\right)\ll e^{\varepsilon \ell}.
$$
Plugging in $B=2\pi$ gives the first approximation given in equation \eqref{eqn:mfixed} and the second follows directly from Stirling's formula.

The terms with $c>1$ contribute to the error by the above argument combined with the fact that the maximal value $g(x_0)$ is asymptotically smaller in this case.
\end{proof}
\begin{remark}
Although one could obtain a bound in general for the terms coming from the non-holomorphic part of the Poincar\'e series, we choose not to do so here because these terms will not play a role the asymptotic of the coefficients of the Faber polynomials.  This occurs because we will only need the above bound when taking linear combinations of harmonic weak Maass forms which are weakly holomorphic modular forms.  Since such forms are holomorphic in the upper half plane, their non-holomorphic parts must necessarily cancel and hence cannot contribute to the asymptotics for the coefficients of the Faber polynomials.
\end{remark}

\section{Coefficients of the Faber Polynomials}\label{section:coeffs}
We have now set up the necessary tools to prove Theorem \ref{thm:coeffvary}.  

\begin{proof}[Proof of Theorem \ref{thm:coeffvary}]
We begin by combining \eqref{eqn:fmeval} and \eqref{eqn:fmFtilde} to obtain
\begin{multline}\label{eqn:cmr}
c_{m,r}\left(\frac{E_6^2(z)}{\Delta(z)}\right)^rE_{k'}(z) = \frac{1}{\Gamma(k)}\left( F(m,2-k;z) + \sum_{n=1}^d b_n F(n,2-k;z)\right)\\
 - \sum_{\substack{0\leq n\leq m-d-1\\ n\neq r}} c_{m,n} \left(\frac{E_6^2(z)}{\Delta(z)}\right)^nE_{k'}(z).
\end{multline}
Since the order of vanishing at $z=i$ on the left hand side is precisely $2r$ (resp. $2r+1$) whenever $k\equiv 2\pmod{4}$ (resp. $k\equiv 0\pmod{4}$), we take the derivative of both sides $2r$ (resp. $2r+1$) times and then evaluate at $z=i$.  We only write down the $k\equiv 2\pmod{4}$ case here.

Since the left hand side is holomorphic in the upper half plane, the right hand side must be as well.  We therefore will only need asymptotics for $F^{(r')}(m',2-w;i)^+$ for some choices of $r'$, $m'$, and $w$. Since the main term in Theorem \ref{thm:growth} came from the holomorphic part, one has the same asymptotic growth for $F^{(r')}(m',2-w,i)^+$ as for $F^{(r')}(m',2-w,i)$.  Since we must take the derivative of each of the $E_6(z)$ occurring on the left hand side exactly once and we may take the derivatives in any order, the derivative of the left hand side equals
\begin{equation}\label{eqn:crmlhs}
(2r)! C_1  C_2^r c_{m,r}.
\end{equation}
We will show that the $2r$-th derivative of the right hand side of \eqref{eqn:cmr} is asymptotically equal to $F^{(r)}(m,2-k;i)$ and then the theorem will follow directly from Theorem \ref{thm:growth}.  

We first consider the terms $\sum_{n=1}^{d} b_n F(n,2-k;z)$.  Choose an orthonormal basis $g_j\in S_k$.  We may write $g_j=\sum_{n=1}^{d} \widetilde{b}_n P(n,k;z)$ for some choice of $\widetilde{b}_n\in \C$, and the work of Bringmann and Ono \cite{BOPNAS} shows that 
$$
G_j(z):=\frac{1}{k-1}\sum_{n=1}^{d} (4\pi n)^{1-k} \overline{\widetilde{b}_n} F(m,2-k;z)
$$
is a lift for $g_j$ (that is, $\xi_{2-k}\left(G_j(z)\right)=g_j(z)$).  Since $\left\{ g_j| j\in \{1,\dots,d\}\right\}$ are orthogonal, it follows that the $G_j$ are independent, and hence give another basis for the space of harmonic weak Maass forms with principal part at most $q^{-d}$.  Therefore
$$
\sum_{n=1}^{d} b_n F(n;2-k;z) = \sum_{j=1}^{d}  c_j G_j(z)
$$
for some constants $c_j$.  Say that $P(m,k;z)=\sum_{j=1}^{d}a_{j,m} g_j$.  Then by integrating $P(m,k;z)$ against itself, one obtains
\begin{equation}\label{eqn:Pmnorm}
\| P(m,k;z)\| = \sum_{n=1}^{d} a_{j,m}^2,
\end{equation}
and 
$$
f_{2-k,m}(z) = F(m,2-k;z)-\sum_{j=1}^{d} a_{j,m} G_j(z),
$$
since $\xi_{2-k}$ acts trivially on the right hand side so that it must be a weakly holomorphic modular form, while $f_{2-k,m}(z)$ is the unique weakly holomorphic modular form with principal part $q^{-m}+O\left(q^{-\ell}\right)$.  Since the bound given in Proposition \ref{prop:mfixedgrowth} is independent of $n$ for $n$ fixed, we obtain the same asymptotic bound for $G_j(z)$, so that 
$$
G_j^{(2r)}(i)^+ \ll (2r)^{2r}=o\left( (2\pi m)^{2r}\right) = O\left( e^{-2\pi m} F^{(2r)}(m,2-k;i)^+\right). 
$$
By \eqref{eqn:Pmnorm}, these terms will contribute to the error as long as $\| P(m,k;z)\|$ grows only polynomially as a function of $m$.  Since the $m$-th Fourier coefficient of $P(m,k;z)$ equals 
$$
\frac{\|P(m,k;z)\|(4\pi m)^{k-1}}{\Gamma(k-1)},
$$ 
we can use the expansion 
$$
1+2\pi i^k \sum_{c>0} \frac{K_{k}(m,m,c)}{c} J_{k-1}\left(\frac{4\pi m}{c}\right)
$$
for the $m$-th coefficient.   Due to a bound of Weil \cite{WeilKloosterman}, the Kloosterman sum grows at most like $m^{\frac{1}{2}}$ as a function of $m$.  In the case $c\ll m$, the $J$-Bessel function decays as a function of $m$, while for $c\gg m$ the $J$-Bessel function grows like $m^{\frac{k-1}{2}}$, so that we obtain polynomial growth in terms of $m$ in both cases.  It follows that 
\begin{equation}\label{eqn:otherterms}
\sum_{n=1}^{d} b_n F^{(2r)}(n,2-k;z)^+ = o\left( F^{(2r)}(m,2-k;i)^+\right).
\end{equation}

We now consider the terms coming from the Faber polynomial with $n\neq r$.  When $r$ is bounded as a function of $m$ we are done, since in that case these terms are bounded by
$$
c_{m,r-1} = O\left(\frac{F^{(2r)}(m,2-k;i)^+}{m^2}\right).
$$
We hence assume that $r\to\infty$.  The $2r$-th derivative of 
$$
c_{m,n} \left(\frac{E_6^2(z)}{\Delta(z)}\right)^nE_{k'}(z)
$$
equals zero at $z=i$ whenever $n>r$, since we cannot take a derivative of each $E_6(z)$ and $E_6(i)=0$.

It remains to bound the terms with $n<r$.  In this case, we keep track of how many times we take the derivative of each term $\frac{E_6^2}{\Delta}(z) = F(1,0;z)+c$ (for some constant $c$) and how many times we take the derivative of $E_{k'}(z)$ when using the product rule repeatedly.  The derivatives of the $E_{k'}(z)$ can easily be shown to satisfy the same bounds (actually, better bounds) as those given in \eqref{eqn:mfixedk0} of Proposition \ref{prop:mfixedgrowth} by writing the Fourier expansion for the Eisenstein series, so, for cosmetic reasons and for clarity of proof, we will treat them universally with the same bound.  Assume that we are taking $r_1$ derivatives of the first term, $r_2$ derivatives of the second term, and so forth.  After reordering to force $r_1\leq r_2\leq \dots \leq r_{n+1}$ with $\sum_{i=1}^{n+1} r_i = 2r$, the number of times we take this many derivatives is counted by the multinomial coefficient
$$
\frac{(2r)!}{r_1!r_2!\cdots \left(r_{n+1}\right)!}.
$$
Thus, using \eqref{eqn:mfixedk0} to bound the derivatives (note that for $r$ bounded the asymptotic is also clearly true), we have the bound 
\begin{equation}\label{eqn:multinom}
\sum_{n=0}^{r-1} c_{m,n} \sum_{\substack{r_1\leq r_2\leq \dots\leq r_{n+1}\\ r_1+\dots + r_{n+1}=2r}} \frac{(2r)!}{r_1!r_2!\cdots r_{n+1}!}\prod_{i=1}^{n+1}\left( \left(\frac{1}{2\pi}+\varepsilon\right)^{r_i} \Gamma\left( r_i-\frac{1}{4}\right) r_i^{\frac{1}{2}}\right).
\end{equation}
We now use Sterling's formula to bound the ratio 
$$
\frac{\Gamma\left( r_i-\frac{1}{4}\right) r_i^{\frac{1}{2}}}{r_i!}\ll r_i^{-\frac{1}{4}}
$$
and the fact that 
$$
\prod_{i=1}^{n+1} \left(\frac{1}{2\pi}+\varepsilon\right)^{r_i} = \left(\frac{1}{2\pi}+\varepsilon\right)^{2r}
$$
to bound the inner sum of \eqref{eqn:multinom} universally, giving the bound of \eqref{eqn:multinom} from above by 
\begin{equation}\label{eqn:univ}
(2r)!\left(\frac{1}{2\pi}+\varepsilon\right)^{2r}\sum_{n=0}^{r-1} c_{m,n}\sum_{\substack{r_1\leq r_2\leq \dots\leq r_{n+1}\\ r_1+\dots + r_{n+1}=2r}} 1
\end{equation}
The inner sum now counts the number of partitions of $2r$ into precisely $n+1$ parts.  We naively bound this by the Hardy and Ramanujan asymptotic 
$$
p(2r)\sim \frac{e^{\pi \sqrt{\frac{4r}{3}}}}{8r\sqrt{3}} 
$$
for the partition function.  Since the $r$ bounded case has already been completed, we may use induction to plug in the asymptotic for $c_{m,n}$.  Since $C_2>1$, which is easily verified by bounding $E_{6}'(i)=1+504\sum_{n=1}^{\infty} n\sigma_5(n)e^{-2\pi n} > 1$ and $\Delta(i)=e^{-2\pi} \prod_{i=1}^{\infty}\left(1-e^{-2\pi n}\right)< e^{-2\pi}$, and 
$$
\frac{(2r)!}{(2n)!} \left(2\pi m\right)^{2n} < \left(2r\right)^{2r-2n} \left(2\pi m\right)^{2n} =O\left( \frac{\left(2\pi m\right)^{2r}}{\pi ^{2r-2n}}\right),
$$
we may bound \eqref{eqn:univ} by 
$$
O\left( \left(\frac{1}{2\pi}+\varepsilon\right)^r e^{\pi \sqrt{\frac{4r}{3}}} F^{(2r)}(m,2-k;i)^+\right).
$$
For $r$ sufficiently large, the factor $\left(\frac{1}{2\pi}+\varepsilon\right)^r e^{\pi \sqrt{\frac{4r}{3}}}$ goes to zero, and the theorem follows.
\end{proof}

\section*{Acknowledgements}
The author would like to thank Kathrin Bringmann for helpful conversation.

\end{document}